\documentclass{amsart}
\usepackage{amssymb,latexsym,amsxtra,amscd,ifthen}
\usepackage{verbatim}
\usepackage[arrow,matrix,cmtip]{xy}

\newcommand {\red} {{\mathrm{red}}}

\DeclareMathOperator{\Spec}{Spec}

\DeclareMathOperator{\Image}{Image}
\DeclareMathOperator{\Supp}{Supp}
\DeclareMathOperator{\coker}{Coker}
\DeclareMathOperator{\Bs}{Bs}
\DeclareMathOperator{\SBs}{SBs}
\DeclareMathOperator{\DB}{{B_{-}}}
\DeclareMathOperator{\AB}{{B_{+}}}

\DeclareMathOperator{\id}{id}

\newcommand {\OO} {{\mathcal O}}

\newcommand {\CC} {{\mathbb C}}

\newcommand {\Z} {{\mathbb Z}}
\newcommand {\ZZ} {{\mathbb Z}}

\numberwithin{equation}{section}

{\theoremstyle{plain}%
  \newtheorem{theorem}[equation]{Theorem}
  \newtheorem{corollary}[equation]{Corollary}
  \newtheorem{proposition}[equation]{Proposition}
  \newtheorem{lemma}[equation]{Lemma}%

{\theoremstyle{remark}

\newtheorem{remark}[equation]{Remark}
}
{\theoremstyle{definition}
\newtheorem{definition}[equation]{Definition}
\newtheorem{example}[equation]{Example}

}

\begin{document}

\title{Arithmetically nef line bundles}

\author{Dennis Keeler}
\address{ Department of Mathematics \\ Miami University \\ Oxford, OH 45056   }
\email{keelerds@miamioh.edu}
\urladdr{http://www.users.miamioh.edu/keelerds}

\begin{abstract}
Let $L$ be a line bundle on a scheme $X$, proper over a field.
The property of $L$ being nef can sometimes be ``thickened,'' allowing
reductions to positive characteristic. We call such line bundles 
\emph{arithmetically nef}. It is known that a line bundle $L$
may be nef, but not arithmetically nef. We show that $L$ is
arithmetically nef if and only if its restriction to its stable base locus
is arithmetically nef.
Consequently, if $L$ is nef and its stable base locus 
has dimension $1$ or less, then $L$ is arithmetically nef.
\end{abstract}

\maketitle

\section{Introduction}

Algebro-geometric theorems over fields of characteristic zero can sometimes
be reduced to theorems over positive characteristic fields.
Perhaps most famously, the Kodaira Vanishing Theorem can be proved 
in this manner, as in \cite[Theorem~6.10]{Illusie}. The main idea
of the reduction is to replace the base field $k$ with 
a finitely generated $\ZZ$-subalgebra $R$ ``sufficiently close'' to $k$.
Objects such as schemes, morphisms, and sheaves are replaced with models
defined over $R$. This process is sometimes called ``arithmetic thickening.''
Some properties of the original objects will be inherited by their
thickened versions, such as ampleness of a line bundle.

However, nefness is not such a property. Langer gave an example of a nef
line bundle that does not have a nef thickening \cite[Section~8]{Langer1}.
Thus on a scheme $X$ proper over a field, we call a line bundle $L$
\emph{arithmetically nef} if $L$ has a nef thickening.
(See \eqref{def:arithmetically-nef} for the exact definition.) 

Arithmetic nefness of a line bundle
was studied briefly in \cite{ArapuraAppendix}, where
it was shown, in characteristic zero,
that $L$ is arithmetically nef if and only if $L$ is
$F$-semipositive (a cohomological vanishing condition). In this paper,
we carry out a more thorough examination, with basic properties 
proven in Section~\ref{sec:arith-nef}. 

We review the stable
base locus in Section~\ref{sec:stable-base-locus}, generalizing the concept
to the case of a Noetherian scheme. We then prove the following
in Section~\ref{sec:main-results}.

\begin{theorem}\label{thm:arith-nef-on-stable-base-locus}
Let $X$ be a proper scheme over a field $k$ with line bundle $L$.
Let $\SBs(L)$ be the stable base locus of $L$.
Then $L$ is arithmetically nef if and only if $L\vert_{\SBs(L)}$ is arithmetically nef.
\end{theorem}

As a corollary, we show that if $L$ is nef and $\dim \SBs(L) \leq 1$,
then $L$ is arithmetically nef (Corollary~\ref{cor:small-stable-base-locus}).
We also show that if $L$ is arithmetically nef,
then any numerically equivalent line bundle is as well
 (Corollary~\ref{cor:arith-nef-numerically-equivalent}).
 
We end in Section~\ref{sec:counterexamples} by reviewing counterexamples of Langer
of line bundles $L$ that are nef, but not arithmetically nef. 
We verify that his positive characteristic counterexample does fail 
our definition of arithmetically nef.


Throughout we work over an arbitrary field $k$ unless otherwise specified.
While the results are of most interest when $k$ has characteristic zero,
this hypothesis will not be necessary for our proofs.

\section{Basic properties of arithmetic nefness}\label{sec:arith-nef}

We must first review \emph{arithmetic thickening}. 
The idea is to approximate a field $k$ by its subalgebras $R_\alpha$ which are
finitely generated over $\Z$. Objects over $\Spec k$, such as finite type schemes,
morphisms, and coherent modules, are replaced with ``thickened'' versions
over $\Spec R_\alpha$. As $k$ is the direct limit 
(that is, a special case of colimit) of the $R_\alpha$,
these thickened objects share many properties with the originals.
When $R_\alpha$ is thickening $k$, we use the subscript $\alpha$
on other associated objects. Note that the $R_\alpha$ form
a directed system, since if $R_\alpha \cup R_\beta \subseteq k$, then
there exists $R_\gamma$ such that $R_\alpha \cup R_\beta \subseteq R_\gamma \subseteq k$.
When $R_\alpha \subseteq R_\beta$, we write $\alpha \leq \beta$.

The basic theory is covered in \cite[Section~6]{Illusie},
with more detail in \cite[$\mathrm{IV}_3$, Section~8]{ega}
and some proofs in \cite[Tag 00QL]{stacks-project}.
(These techniques work more generally in the case of finite presentation,
but since we are working with Noetherian rings, this is the same as finite type.)

Since localizing at $0 \neq f \in R_\alpha$ yields $(R_\alpha)_f$ with
$R_\alpha \subseteq (R_\alpha)_f \subset k$, we can also replace
$\Spec R_\alpha$ with appropriate basic open subsets.
The appendices of \cite{Gortz} contain a large list (with references) of 
properties that hold on the fibers over open subsets of $\Spec R_\alpha$.

Let $S_\alpha=\Spec R_\alpha$. We will always choose $R_\alpha$ large enough so that 
$f_\alpha:X_\alpha \to S_\alpha$ is proper \cite[$\mathrm{IV}_3$, 8.10.5]{ega},
and hence closed. Since the inclusion $R_\alpha \hookrightarrow k$
factors through the fraction field of $R_\alpha$, the generic point
of $S_\alpha$
must be in the image of $f_\alpha$. Hence $f_\alpha$ is surjective.

Some authors also require that $X_\alpha \to S_\alpha$ be flat.
This can always be accomplished for large enough $\alpha$
\cite[$\mathrm{IV}_3$, 11.2.6]{ega}.
We can also guarantee that $L_\alpha$ 
is invertible \cite[$\mathrm{IV}_3$, 8.5.5]{ega}.

The following lemma summarizes some properties that can be preserved in a thickening.
Since many of our proofs reference \cite{ega} or \cite{hartshorne2}, 
we note that their definitions of projective morphism
coincide when the target has an ample line bundle \cite[Tag 01W9, 087S]{stacks-project},
and this will always be the case for this paper.

\begin{lemma}\label{lem:thickened-smooth}
Let $X$ be an integral scheme, 
smooth and projective over an algebraically closed field $k$.
Then there exists a thickening $(X_\alpha, R_\alpha)$ such that for all
$\beta \geq \alpha$, we have
\begin{enumerate}
\item\label{thick1} $f_\beta: X_\beta \to \Spec R_\beta$ is smooth and projective,
\item\label{thick2} For every $s \in \Spec R_\beta$, the fiber $X_s$ is geometrically
integral over the residue field $k(s)$ and all fibers have the same dimension,
\item\label{thick3} For every $s \in \Spec R_\beta$, the induced map $f_s: X_s \to \Spec k(s)$
is smooth and projective.
\end{enumerate}
\end{lemma}
\begin{proof}
Choose a thickening $(X_\alpha, R_\alpha)$ with 
$f_\alpha:X_\alpha \to S_\alpha=\Spec R_\alpha$
smooth and projective \cite[$\mathrm{IV}_{4,3}$, 17.7.8, 8.10.5]{ega}.
Smooth and projective are stable under base change \cite[Tag 01VB, 01WF]{stacks-project},
which gives us claims \eqref{thick1} and \eqref{thick3}.

Let $k_\alpha$ be the fraction field of $R_\alpha$. Note $k_\alpha \subseteq k$.
 Then
$X_\alpha \times_{R_\alpha} k_\alpha$ is geometrically integral
by definition, since taking the fiber product with $\Spec k$
yields the integral scheme $X$.

Since $f_\alpha$ is proper and flat, 
we can replace $R_\alpha$ with a localization by a non-zero element
and assume that for every $s \in \Spec R_\beta$, the fibers $X_s$ 
of ${f_\alpha}$ are geometrically integral,
 of the same dimension 
\cite[$\mathrm{IV}_3$, 12.2.1]{ega}. 

Now if $R_\alpha \subseteq R_\beta$ and $s' \in \Spec R_\beta$, then
$k(s')$ is a field extension of $k(s)$ for some $s \in \Spec R_\alpha$.
Thus the fiber of $f_\beta$ over $s'$ is still geometrically integral.
The base change of $\Spec k(s') \to \Spec k(s)$ also 
preserves the dimension of the fiber \cite[Tag 02FY]{stacks-project}.
Hence we have \eqref{thick2}.
\end{proof}


We also must review the concept of relative nefness.
For a review of nefness and intersection theory in general,
see  the seminal paper \cite{kleiman} (working
over an algebraically closed field) or \cite[Chapter~VI.2]{kollar}
(working over an arbitrary field).
	
Let $S$ be a noetherian scheme, let
$f\colon X \to S$ be a proper morphism, and let $L$ be a line bundle on $X$.
For each $s \in S$, let $L_s$ be the restriction of $L$ to the fiber $X_s$.
Recall that $L$ is \emph{$f$-nef} if $L_s$ is nef for every closed
$s \in S$ (see, for instance, \cite[Definition~2.9]{keeler}).
If $S$ is affine, then the property of $L$ being $f$-nef does not
depend on $f$, so we may simply say that $L$ is nef \cite[Proposition~2.15]{keeler}.

We now define the main concept of arithmetically nef. While
the concept is most useful in characteristic $0$, there is no harm in 
using an arbitrary field.
(In \cite[Definition~3.12]{ArapuraAppendix}, when $k$ had positive characteristic, 
the line bundle $L$ was defined to be arithmetically nef if $L$ was nef,
but this was for convenience.)  
We do not assume that $k$ is algebraically closed.
	
\begin{definition}\label{def:arithmetically-nef}
Let $X$ be a proper scheme over a field $k$, and let $L$ be a line bundle on $X$.
Then $L$ is \emph{arithmetically nef} if there exists a thickening
$(X_\alpha \to \Spec R_\alpha, L_\alpha)$ such that $L_\alpha$ is nef.
\end{definition}

\begin{remark}\label{rem:char-p}
We have not insisted that the field $k$ have characteristic $0$, though
this is usually the case when applying arithmetic thickening.
When $k$ is algebraic over a finite field, an arithmetic thickening
just yields a subfield $R_\alpha$. Indeed, $R_\alpha \subseteq k$
will be algebraic over $\ZZ/p\ZZ$ and hence $R_\alpha$
is a field 
\cite[{Tag 00GS}]{stacks-project}. In this case, nef and arithmetically nef
are equivalent. On the other hand, Example~\ref{ex:char-p}
reviews a nef, but not arithmetically nef, line bundle when
the base field is $\overline{\mathbb{F}}_2(t)$.
\end{remark}

\begin{remark}\label{rem:once-nef-always-nef}
The property of $L_\alpha$ being nef is stable under base change
of $R_\alpha$ \cite[Lemma~2.18]{keeler}.
Thus if $L$ is arithmetically nef, then $L_\alpha$ will be nef
on every fiber of a certain thickening \cite[Lemma~2.18]{keeler}. In particular,
if $L$ is arithmetically nef, then $L$ is nef.

Further, if $L_\alpha$ is nef, then $L_\beta$ is nef for any $\beta \geq \alpha$.
That is, once one thickening works, all subsequent thickenings will also work.
\end{remark}

Like nefness, the property of being arithmetically nef behaves well under
pullbacks. If $i:Y \to X$ is a closed immersion, we write $L\vert_Y= i^*L$.

\begin{lemma}\label{lem:arith-nef-surjective}
Let $L$ be a line bundle on a proper scheme $X$ over a field $k$, and
let $f\colon X' \to X$ be a proper morphism. 
\begin{enumerate}
\item If $L$ is arithmetically nef, then $f^*L$ is arithmetically nef, and
\item if $f$ is surjective and $f^*L$ is arithmetically nef, then
$L$ is arithmetically nef.
\end{enumerate}
In particular, if $L$ is arithmetically nef, then $L\vert_Y$ is arithmetically nef
for any closed subscheme $Y \subseteq X$.
\end{lemma}
\begin{proof}
Choose a thickening
$(f_\alpha\colon X'_\alpha \to X_\alpha, L_\alpha)$ such
that $f_\alpha$ is proper and also surjective if $f$ is surjective
\cite[$\mathrm{IV}_3$, 8.10.5]{ega}. If $L$ is arithmetically nef, then upon
choosing a larger $\alpha$,
we may assume $L_\alpha$ is nef,
and so ${f}_\alpha^* {L}_\alpha$ is nef \cite[Lemma~2.17]{keeler}.
Hence $f^*L$ is arithmetically nef. Now if $f^*L$ is arithmetically nef,
then $L$ is arithmetically nef by a similar argument \cite[loc. cit.]{keeler}.
\end{proof}

Similar to ampleness, the concept of arithmetically nef depends only on 
the reduced, irreducible components.

\begin{lemma}\label{lem:irreducible-components}
Let $L$ be a line bundle on a proper scheme $X$ over a field $k$. 
Let $X_i$, $i=1, \dots, n$ be the irreducible components of $X$
and $L_i$ be the restriction of $L$ to $X_i$. Then
\begin{enumerate}
\item $L$ is arithmetically nef if and only if $L_{\red}$
(the restriction of $L$ to $X_{\red}$) is arithmetically nef.
\item $L$ is arithmetically nef if and only if $L_i$ 
is arithmetically nef for all $i$.
\end{enumerate}
\end{lemma}
\begin{proof}
The first claim follows immediately from $i: X_{\red} \to X$
and Lemma~\ref{lem:arith-nef-surjective}. 
For the second claim, we may assume $X$ is reduced and use the
natural surjection $f: \coprod_i {X_i} \to X$
from the disjoint union of the $X_i$ to $X$.
\end{proof}

Arithmetically nef is preserved by base change, allowing reduction to
the case of an algebraically closed field.
We use a form of faithfully flat descent for one direction
of the proof.

\begin{lemma}\label{lem:arith-nef-base-change}
Let $X$ be proper over a field $k$, let $k'$ be a field extension of $k$,
and let $L$ be a line bundle on $X$. Then $L$ is arithmetically nef
on $X$ if and only if $L \otimes_k k'$ is arithmetically nef
on $X \times_k k'$.
\end{lemma}
\begin{proof}
Let $({X}_\alpha, R_\alpha, {L}_\alpha)$ be an arithmetic thickening 
of $(X, k, L)$ with ${L}_\alpha$ invertible.

Suppose $L$ is arithmetically nef. So we may assume that ${L}_\alpha$
is nef on ${X}_\alpha$. Since $R_\alpha \subseteq k \subseteq k'$,
 we have natural isomorphisms
\[
{X}_\alpha \times_{R_\alpha} k' 
\cong ({X}_\alpha \times_{R_\alpha} k) \times_k k' \cong X \times_k k'.
\]
We see that $({X}_\alpha, R_\alpha, {L}_\alpha)$ is also an
arithmetic thickening of  $(X \times_k k', k', L \otimes_k k')$.
Hence $L \otimes_k k'$ is arithmetically nef.

Now suppose $L \otimes_k k'$ is arithmetically nef. 
Due to Remark~\ref{rem:once-nef-always-nef}
and the uniqueness in \cite[Proposition~6.2]{Illusie},
we can choose a finitely generated $\Z$-algebra $R_\beta$ such that
$R_\alpha \subseteq R_\beta \subseteq k'$ and $({X}_\alpha \times_{R_\alpha} R_\beta,
R_\beta, {L}_\alpha \otimes_{R_\alpha} R_\beta)$ is an arithmetic thickening
of $(X\times_k k', k', L \otimes_k k')$, where
${L}_\alpha \otimes_{R_\alpha} R_\beta$ is nef.

Since $R_\alpha \subseteq R_\beta$ and both $R_\alpha, R_\beta$ are finite type over $\Z$,
we have $R_\beta$ of finite type over $R_\alpha$.
By Generic Freeness \cite[Tag 051R]{stacks-project}, we
can choose $f \in R_\alpha$ such that $(R_\beta)_f$ is a free $(R_\alpha)_f$-module.
Thus the map $\Spec (R_\beta)_f \to \Spec (R_\alpha)_f$ is flat and surjective.
Note that ${L} \otimes_{R_\alpha} (R_\beta)_f$ is still nef
and thus ${L}_\alpha \otimes_{R_\alpha} (R_\alpha)_f$ is nef
\cite[Lemma~2.18]{keeler}.
Since $({X}_\alpha \otimes_{R_\alpha} (R_\alpha)_f, 
(R_\alpha)_f, {L}_\alpha \otimes_{R_\alpha} (R_\alpha)_f)$
is an arithmetic thickening of $(X, k, L)$, we
have that $L$ is arithmetically nef.
\end{proof}

Due to the ability to change fields, we can check arithmetic nefness by 
just checking nefness over open subsets of a fixed base $\Spec R_\alpha$.

\begin{corollary}\label{cor:open-sets}
Let $X$ be proper over a field $k$ with line bundle $L$.
Let $(X_\alpha, R_\alpha, L_\alpha)$ be an arithmetic thickening
such that $L_\alpha$ is invertible and $X_\alpha \to \Spec R_\alpha$
is proper. Then $L$ is arithmetically nef if and only if
there exists $0 \neq f \in R_\alpha$ such that $L_\beta$ is nef,
where $R_\beta = (R_\alpha)_f$. 
\end{corollary}
\begin{proof} Of course if $L_\beta$ is nef, then $L$ is arithmetically
nef by definition. 

So suppose $L$ is arithmetically nef. Let $k_\alpha$ be the fraction field
of $R_\alpha$. Since $R_\alpha \subseteq k$, we have $k_\alpha \subseteq k$.
Let $M_\alpha = L_\alpha \otimes_{R_\alpha} k_\alpha$.
By Lemma~\ref{lem:arith-nef-base-change}, we have that
$M_\alpha$ is arithmetically nef because
$L \cong M_\alpha \otimes_{k_\alpha} k$ is arithmetically nef.

By definition, there exists a finite type $\Z$-algebra $R_\beta \subseteq k_\alpha$ 
such that there is a nef thickening of $M_\alpha$ over $R_\beta$. 
By Remark~\ref{rem:once-nef-always-nef}, we can choose $R_\beta \supseteq R_\alpha$,
and so $R_\beta$ is finitely generated over $R_\alpha$ and
$R_\alpha \subseteq R_\beta \subseteq k_\alpha$. 
Since $k_\alpha$ is the fraction field of $R_\alpha$, we must have 
$R_\beta = (R_\alpha)_f$ for some non-zero $f \in R_\alpha$.
\end{proof}

The above corollary makes it easier to show that a line bundle is \emph{not}
arithmetically nef. It cannot occur that $R_\alpha$ farther from $k$
does not give nefness, while a closer approximation does. 
For example, we have the following.

\begin{corollary}\label{cor:non-arith-nef}
Let $X$ be proper over a field $k$ with line bundle $L$.
Let $(X_\alpha, R_\alpha, L_\alpha)$ be an arithmetic thickening
such that $L_\alpha$ is invertible and $X_\alpha \to \Spec R_\alpha$
is proper.

Suppose that for all closed $s \in \Spec R_\alpha$, 
the line bundle $L_{\alpha,s}$ is not nef on the fiber $X_{\alpha,s}$.
Then $L$ is not arithmetically nef.
\end{corollary}
\begin{proof}
Let $S=\Spec R_\alpha$.
The ring $R_\alpha$ is Jacobson
because it is finitely generated over $\Z$ \cite[Tag 00GC]{stacks-project}.
Thus $S$ is a Jacobson space and the closed points of any open subset $U$
are closed in $S$
\cite[Tag 00G3, 005X]{stacks-project}.
So $L_\alpha\vert_U$ cannot be relatively nef.
Thus $L$ is not arithmetically nef by Corollary~\ref{cor:open-sets}.
\end{proof}

We end the section with a few more simple observations.
Being arithmetically nef also behaves well under tensor product.

\begin{lemma}\label{lem:arith-nef-tensor}
Let $X$ be a proper scheme over a field $k$ with line bundles $L,M$. Then
\begin{enumerate}
\item $L$ is arithmetically nef, if and only if $L^n$ is arithmetically nef for
all $n > 0$, if and only if $L^n$ is arithmetically nef for some $n > 0$, and
\item If $L$ and $M$ are arithmetically nef, then $L \otimes M$ is arithmetically nef.
\end{enumerate}
\end{lemma}
\begin{proof} These statements follow immediately from the definition of nef.
\end{proof}

In the projective case, arithmetically nef has a strong connection to ample.

\begin{proposition}
Let $X$ be a projective scheme over a field $k$, let $H$ be an ample
line bundle, and let $L$ be a line bundle. Then $L$ is arithmetically nef
if and only if there exists an arithmetic thickening ${X}_\alpha \to \Spec R_\alpha$
such that ${H}_\alpha \otimes {L}_\alpha^n$ is ample for all $n > 0$.
\end{proposition}
\begin{proof} Choose an arithmetic thickening $f:{X}_\alpha \to \Spec R_\alpha$
such that $f$ is projective \cite[$\mathrm{IV}_3$, 8.10.5]{ega}, 
${H}_\alpha$ is an ample line bundle \cite[$\mathrm{IV}_3$, 9.6.4]{ega},
and ${L}_\alpha$ is a line bundle \cite[$\mathrm{IV}_3$, 9.4.7]{ega}.
The result then follows from \cite[Proposition~2.14]{keeler}.
\end{proof}

\section{Stable base locus}\label{sec:stable-base-locus}

We now consider some basic properties of the \emph{stable base locus}
of a line bundle $L$ on a scheme $X$. While the concept is usually defined when 
$X$ is proper over a field, it will be useful to consider any Noetherian scheme.
Hence, we generalize the definition via the following proposition.
We take supports of coherent sheaves to have their reduced induced structure.

\begin{proposition}\label{prop:base-locus-exists}
Let $X$ be a Noetherian scheme with line bundle $L$. Consider exact sequences
of the form
\[
\oplus_{i=1}^n \OO_X {\to} L \to F \to 0,
\]
for some $n \geq 1$.
Then there exists a reduced closed subscheme $Z$ of $X$
such that $Z$ is the minimum support for any such cokernel $F$.
That is, $Z = \Supp(\coker u)$ for some 
 $u: \oplus_{i=1}^n \OO_X \to L$, and for any $v: \oplus_{i=1}^m \OO_X \to L$
 (with $m \geq 1$),
we have $Z \subseteq \Supp(\coker v)$.
\end{proposition}
\begin{proof}
Note that all sheaves involved are coherent, and hence have closed support
\cite[Corollary~7.31]{Gortz}.
Let 
\[
S = \{ Y \mid \exists n \geq 1 \text{ and } u:\oplus_{i=1}^n \OO_X \to L \text{ such that } 
Y = \Supp(\coker(u)) \}.
\] 
Since $X$ is Noetherian, $S$ has minimal elements.

Let $Z_1, Z_2$ be two minimal elements of $S$. So there exists $n_i \geq 1$
and $u_i: \oplus_{i=1}^{n_i} \OO_X \to L$
with $Z_i = \Supp(\coker(u_i))$. Define $v: \oplus_{i=1}^{n_1 + n_2} \OO_X \to L$
by $v(a_1 \oplus a_2) = u_1(a_1) + u_2(a_2)$ for $a_i  \in \oplus_{i=1}^{n_i} \OO_X$.

Support will not change when tensoring all maps by $L^{-1}$, so 
we can abuse notation and consider $I_i = \Image(u_i)$ to be a sheaf of ideals.
Since $\Image(v) = I_1 + I_2$, we have $\coker(v) = \OO_X/(I_1 + I_2)$.
Thus $\Supp(\coker(v))$ is the scheme theoretic intersection
of $Z_1$ and $Z_2$ \cite[Tag 0C4H]{stacks-project}  (with
reduced induced structure). Since $Z_1, Z_2$ are minimal,
we must have $Z_1 = Z_2$. Hence the \emph{minimum}
$Z$ exists.
\end{proof}

\begin{definition}\label{def:base-locus}
Let $X$ be a Noetherian scheme with line bundle $L$. The \emph{base locus}
of $L$, denoted $\Bs(L)$, is the closed subscheme $Z$
of Proposition~\ref{prop:base-locus-exists}.
\end{definition}

Our definition of base locus is similar to that of
\cite[Section~1.1.B]{PAG}, in the case where $X$ is projective over $\CC$.
That is, $\Bs(L)$ can be interpreted as the support of the cokernel
of the natural map $H^0(X,L) \otimes_k L^{-1} \to \OO_X$.
Another common formulation is that of \cite[Section~2.7]{Birkar}.
Working with $X$ a projective scheme over an arbitrary field $k$, the paper defines
\begin{equation}\label{eq:usual-base-locus}
\Bs(L) = \{x \in X \mid s \text{ vanishes at } x \text{ for all } s \in H^0(X,L) \}.
\end{equation}

\begin{proposition}
Let $X$ be a scheme, proper over a field $k$,
with line bundle $L$. Then the definitions \eqref{def:base-locus},
 \eqref{eq:usual-base-locus}, and \cite[Section~1.1.B]{PAG}
  define the same reduced closed
subscheme $\Bs(L)$.
\end{proposition}
\begin{proof}
Every $s \in H^0(X,L)$ defines a $\OO_X$-module
homomorphism $\phi_s: \OO_X \to L$ by $\phi_s(1)=s$.
Conversely, any homomorphism $\nu: \OO_X \to L$ 
gives a global section $\nu(1) \in H^0(X,L)$
\cite[Tag 01AL]{stacks-project}.
Thus we can define $u: H^0(X,L) \otimes_k \OO_X \to L$
by $u(s \otimes a) = \phi_s(a) = as$.

Let $Y$ be the base locus defined by Equation~\ref{eq:usual-base-locus}
and let $W$ be the base locus defined by \cite[Section~1.1.B]{PAG}.
If $x \in Y$, then $s(x) = 0$ for all $s \in H^0(X,L)$. 
This means $s \otimes k(x) = 0$ where $k(x)$ is the residue field at $x$.
So $u \otimes \id_{k(x)} = 0$ and thus $u$ cannot be onto the stalk $L_x$.
So $\coker(u)_x \neq 0$ and hence $x \in W$. On the other hand, if
$x \not\in Y$, then there exists $s \in H^0(X,L)$ such that $s \otimes k(x) \neq 0$.
We have $s_x$ invertible in $L_x \cong \OO_{X,x}$. So $u$ is onto the stalk $L_x$
and $x \not\in W$. Thus $Y=W$.

Let $Z$ be the base locus defined by Definition~\ref{def:base-locus}, and
$\oplus_{i=1}^n \OO_X {\to} L \to F \to 0$ be an exact sequence with $Z = \Supp(F)$.
Let $\phi_j: \OO_X \to L$ be the induced homomorphism from the $j$th component
of $\oplus_{i=1}^n \OO_X$. 
If $x \not\in Z$, then $\oplus_{i=1}^n \OO_{X,x} {\to} L_x$ is surjective.
Thus there exists $j \in \{1, \dots, n\}$ such that $\phi_j(1)_x$ is not
an element of the maximal submodule of $L_x$. Hence
$\phi_j(1) \otimes k(x) \neq 0$. In other words,  the global section
$\phi_j(1) \in H^0(X,L)$ does not vanish at $x$. Thus $x \not\in Y$.

On the other hand, $Z \subseteq Y=Z$ by the minimality in the definition of $Z$.
Hence $Z=Y=W$, as desired.
\end{proof}

Proving the existence of the stable base locus is similar to the projective case.

\begin{proposition}\label{prop:stable-base-locus-exists} Let $X$ be a Noetherian scheme with line bundle $L$.
For any $p, q > 0$, we have $\Bs(L^{pq}) \subseteq \Bs(L^p)$.
Thus there exists $m > 0$ such that for all $n > 0$, $\Bs(L^m) \subseteq \Bs(L^n)$
and $\Bs(L^m) = \Bs(L^{mn})$.
\end{proposition}
\begin{proof}
It suffices to prove the first claim for $p=1$. Let $Y = \Bs(L)$ and $j > 0$,
$u: \oplus_{i=1}^j \OO_X \to L$ such that $Y = \Supp(\coker(u))$.
Taking $U = X \setminus Y$, we have that $u\vert_U : \oplus_{i=1}^m \OO_U \to L\vert_U$
is surjective.
Thus $(u\vert_U)^{\otimes q} : (\oplus_{i=1}^m \OO_U)^{\otimes q} \to L^q\vert_U$
is surjective. Therefore $\Bs(L^q) \subseteq \Supp(\coker(u^{\otimes q})) \subseteq \Bs(L)$
by the minimality of $\Bs(L^q)$.

The remaining claims follow by the Noetherian property on closed sets. 
See for example \cite[Proposition~2.1.21]{PAG}.
\end{proof}

\begin{definition}\label{def:stable-base-locus}
Let $X$ be a Noetherian scheme with line bundle $L$. The \emph{stable base locus}
of $L$, denoted $\SBs(L)$, is the closed subscheme $Y$
of Proposition~\ref{prop:stable-base-locus-exists}.
\end{definition}

We work over a general Noetherian scheme so that we need not worry about
the behavior of $H^0(X,L)$ when $X$ is not reduced or not irreducible.
Also, right exact sequences can be thickened
\cite[$\mathrm{IV}_3$, 8.5.6]{ega} and 
these sequences
behave well under pullbacks, leading to the following lemma.

\begin{lemma}\label{lem:base-locus-pullback}
Let $X,Y$ be Noetherian schemes, $f:Y \to X$ a morphism,
and $L$ a line bundle on $X$. Then
\[
\Bs(f^*L) \subseteq f^{-1}(\Bs(L)), \qquad \SBs(f^*L) \subseteq f^{-1}(\SBs(L)).
\]
In particular, if $Y$ is a locally closed subscheme of $X$, then
\[
\Bs(L\vert_Y) \subseteq Y\cap \Bs(L), \qquad \SBs(L\vert_Y) \subseteq Y\cap \SBs(L).
\]
\end{lemma}
\begin{proof}
Let $Z = \Bs(L)$ with associated right exact sequence
\[
\oplus_{i=1}^n \OO_X \to L \to F \to 0.
\]
Pulling back by $f$ is right exact \cite[Remark~7.9]{Gortz}, so we have
\[
\oplus_{i=1}^n \OO_Y \to f^*L \to f^*F \to 0.
\]
By \cite[Tag 056J]{stacks-project}, we have $\Supp(f^*F) = f^{-1}(\Supp(F)) = f^{-1}(Z)$.
Then by the minimality in the definition of base locus, $\Bs(f^*L) \subseteq f^{-1}(Z)$.

The claim regarding stable base locus then follows easily.
\end{proof}

\begin{remark}\label{rem:stable-base-locus-equality}
Let $k \subseteq k'$ be an extension of fields and let $X$ be a projective
scheme over $k$, and $L$ a line bundle on $X$. Then 
taking $f: X \times_k k' \to X$, Birkar shows $\Bs(f^*L) = f^{-1}(Bs(L))$
and $\SBs(f^*L) = f^{-1}(\SBs(L))$ \cite[Section~2.7]{Birkar}.
The proof only uses the finite dimensionality of $H^0(X,L)$
and $H^0(X \times_k k', f^*L) \cong H^0(X,L) \otimes_k k'$, so this
works for proper schemes as well.

Fujita shows that when $Y,X$ are proper over an algebraically closed field,
with $X$ normal and integral, $Y$ irreducible, and $f:Y \to X$ surjective,
then $\SBs(f^*L) = f^{-1}(\SBs(L))$ \cite[Theorem~1.20]{Fujita}. 
On the other hand, \cite[1.21]{Fujita} gives an example of proper
containment when $X$ is a non-normal rational curve and $Y$ is its normalization.
\end{remark}

\section{Main results}\label{sec:main-results}

As stated, a line bundle can be nef, yet not arithmetically nef. However,
in this section we examine some properties that guarantee arithmetic nefness.
Weaker versions of some of these results were mentioned in 
\cite[Appendix]{ArapuraAppendix} without proof.

\begin{proposition}\label{prop:numerically-trivial-arith-nef}
Let $X$ be a proper scheme over a field $k$ with line bundle $L$. If $L$ is numerically trivial
(that is, $L$ and $L^{-1}$ are nef),
then $L$ and $L^{-1}$ are arithmetically nef (and hence arithmetically numerically trivial).
\end{proposition}
\begin{proof}
We need only show that $L$ is arithmetically nef.

First, nefness is stable under both field extension and descent 
by \cite[Lemma~2.18]{keeler}, as is arithmetic nefness by Lemma~\ref{lem:arith-nef-base-change}.
So we may assume that $k$ is algebraically closed.
Second, we may similarly assume that $X$ is integral by 
definition of nef and Lemma~\ref{lem:irreducible-components}.
Finally, any pullback of $L$ by a proper morphism is also
numerically trivial \cite[Lemma~2.17]{keeler}. Thus by Lemma~\ref{lem:arith-nef-surjective}
we may replace $X$ with a Chow cover and thus assume that $X$ is projective.
Via Alteration of Singularities, we can even assume that
$X$ is smooth over $k$ \cite[Theorem~4.1]{deJong}.
Let $S_\alpha = \Spec R_\alpha$.
By Lemma~\ref{lem:thickened-smooth}, we can assume all thickenings
$f_\alpha: X_\alpha \to S_\alpha$ are smooth and projective, with
geometrically integral, constant dimensional fibers.

Let $H$ be an ample line bundle on $X$. There exists a thickening
where $H_\alpha$ is ample \cite[$\mathrm{III}_1$, 4.7.1]{ega},
and so any subsequent $H_\beta$ is also ample \cite[Tag 0893]{stacks-project}.
For any $s \in S_\alpha$, let $H_s, L_s$ be the pullbacks of $H_\alpha, L_\alpha$
to the fiber $X_s$ of $f_\alpha$ over $s$. (We omit the $\alpha$ since the thickening
is not changing.) 
Since $f_\alpha$ is smooth, and hence flat \cite[Tag 01VF]{stacks-project}, 
and $S_\alpha$ is irreducible, we have 
that the intersection numbers $(L_{\overline{s}}^r.H_{\overline{s}}^{\dim X - r})$
are constant as $s$ varies, and $r=0, \dots, \dim X$ is fixed \cite[VI.2.9]{kollar}.
(Here the exponent means self-intersection, not tensor product.)


Let $L_{\overline{s}} = L_s \otimes \overline{k(s)}, 
H_{\overline{s}} = H_s \otimes \overline{k(s)}$ on
the smooth, (geometrically) integral scheme $X_{\overline{s}} = X_s \times_{k(s)} \overline{k(s)}$.
Tensoring cohomology of $H_s^m \otimes L_s^n$ (for fixed $m,n \in \ZZ$) with
the algebraic closure of $k(s)$ will not change the value of the Euler
characteristic \cite[Tag 02KH]{stacks-project}.
The intersection numbers $(L_{\overline{s}}^r.H_{\overline{s}}^{\dim X - r})$,
being derived from the Euler characteristic \cite[Section~I.2]{kleiman},
are thus also constant with respect to $s$.

If $\dim X = 0$, then $X$ is just a point and $L\cong \OO_X$ thickens
to a relatively numerically trivial line bundle.
If $\dim X = 1$, then $(L_{{s}}. X_{s}) = 0$
for all $s$. So again, $L_s$ is nef on every fiber and hence arithmetically nef.

Note that $L_{\overline{s}}$ is numerically trivial if and only if
the same holds for $L_s$ \cite[Lemma~2.18]{keeler}.
If $\dim X \geq 2$, then
${L}_{\overline{s}}$ is numerically trivial if and only if
\[
({L}_{\overline{s}}.{H}_{\overline{s}}^{\dim X -1}) 
= ({L}_{\overline{s}}^2.{H}_{\overline{s}}^{\dim X - 2}) =  0
\]
by \cite[p.~305, Corollary~3]{kleiman}, as a consequence of the Hodge Index Theorem.
The intersection numbers equal $0$ at the generic point and hence
for all $s$. (See \cite[p.~335]{kleiman} for a similar argument.)
Thus $L$ is arithmetically nef, as is $L^{-1}$.
\end{proof}

\begin{remark} Instead of changing to an algebraically closed field,
one could use a relativized version of Alteration of Singularities,
as in \cite[Lemma~E1.3]{keeler-cor}, but replacing $R_\alpha$
with a finite extension might be required.
\end{remark}

We now have that arithmetic nefness is preserved under
numerical equivalence. We write $L \equiv L'$ if $L,L'$
are numerically equivalent. That is, $(L.C) = (L'.C)$ for all integral curves $C$.

\begin{corollary}\label{cor:arith-nef-numerically-equivalent}
Let $X$ be a proper scheme over a field $k$ with line bundles
$L,L'$. Suppose $L \equiv L'$. Then
$L$ is arithmetically nef if and only if $L'$ is arithmetically nef.
\end{corollary}
\begin{proof} This follows from Lemma~\ref{lem:arith-nef-tensor} 
and Proposition~\ref{prop:numerically-trivial-arith-nef}.
\end{proof}

At least when $X$ is a curve, nef and arithmetically nef are always the same.

\begin{corollary}\label{cor:arith-nef-on-curves}
Let $X$ be a proper scheme over a field $k$ with $\dim X \leq 1$. 
If $L$ is a nef line bundle,
then $L$ is arithmetically nef.
\end{corollary}
\begin{proof} The case of $\dim X = 0$ is trivial since any line bundle is ample.
Using the same argument as in the proof of
Proposition~\ref{prop:numerically-trivial-arith-nef},
we may assume that $X$ is a smooth, integral projective curve
over an algebraically closed field. Then $L$ is either numerically trivial
or ample, and hence arithmetically nef by 
Proposition~\ref{prop:numerically-trivial-arith-nef}
or \cite[$\mathrm{III}_1$, 4.7.1]{ega}.
\end{proof}

We now examine connections between nefness of $L$ and the (stable) base locus
$B = \SBs(L)$.
Theorems in this vein already exist. For example, when $X$ is proper over
an algebraically closed field,
\cite[Theorem~1.10]{Fujita} shows that if $L\vert_{\Bs(L)}$ is ample,
then $L$ is semiample. On the other hand, \cite[Example~1.16]{Fujita}
shows $L\vert_{\Bs(L)}$ can be semiample, yet $L$ is not semiample.

Recently, Birkar showed that for any projective scheme over a field $k$,
if $L$ is nef,
then there exists a subscheme $Z$ with $Z_{\red}$ equal to the \emph{augmented base locus}
$\AB(L)$, such that $L$ is semiample if and only if $L\vert_Z$ is semiample 
\cite[Theorems~1.4, 1.5]{Birkar}. See \cite[Section~1]{Birkar} for a
history of similar results.

Also, the \emph{diminished base locus} or \emph{non-nef locus} $\DB(L)$
was studied in \cite{ELMMP,Mustata}. One has that $\DB(L) \subseteq \SBs(L)$
and $L$ is nef if and only if $\DB(L) = \emptyset$.
(Note that it is possible for $\DB(L)$ to not be closed \cite{Les}.)
The following theorem fits with these properties of $\DB(L)$ since
the non-nef locus is contained in the stable base locus.

\begin{theorem}\label{thm:nef-on-stable-base-locus}
Let $X$ be a proper scheme over a field $k$ with line bundle $L$.
Then $L$ is nef if and only if $L\vert_{\SBs(L)}$ is nef.
\end{theorem}
\begin{proof}
Since nefness and the stable base locus 
are stable under field extension and descent,
by \cite[Lemma~2.18]{keeler} and Remark~\ref{rem:stable-base-locus-equality},
we can assume $k$ is algebraically closed. This will 
match the hypotheses of \cite{Fujita,kleiman}.

If $L$ is nef, then clearly $L\vert_{\SBs(L)}$ is nef.
So suppose $L\vert_{\SBs(L)}$ is nef. Let $C \subseteq X$
be an integral curve. If $C \subseteq \SBs(L)$, then 
$(L.C) = (L\vert_{\SBs(L)}.C)$ \cite[p.~298, Proposition~5]{kleiman},
and $(L\vert_{\SBs(L)}.C) \geq 0$ by hypothesis.

If $C \not\subseteq \SBs(L)$, then $C \cap \SBs(L)$ is a finite set.
Thus $\SBs(L\vert_C)$ is a finite set since 
$\SBs(L\vert_C) \subseteq C \cap \SBs(L)$ by Lemma~\ref{lem:base-locus-pullback}.
But then $\SBs(L\vert_C) = \emptyset$ \cite[Corollary~1.14]{Fujita}.
So $L\vert_C$ is semiample and thus we may replace $L$ with a
multiple and assume $L\vert_C$ is generated
by global sections. Thus $L\vert_C$ is the pullback of a very ample line bundle
on some projective space. So we have
$(L.C) = (L\vert_C.C) \geq 0$ by \cite[p.~303, Proposition~1]{kleiman}.
This shows that $L$ is nef.
\end{proof}

\begin{remark}
The above theorem makes the nef hypothesis unnecessary in
\cite[Theorem~1.5]{Birkar}.
We have $\SBs(L) \subseteq \AB(L)$ by \cite[Example~1.16]{ELMMP},
\cite[Section~2]{CMM}, \cite[Section~2.7]{Birkar}.
If
$L\vert_Z$ (with $Z_{\red} = \AB(L)$) is semiample, then 
$L\vert_{\SBs(L)}$ is semiample and hence nef.
So by Theorem~\ref{thm:nef-on-stable-base-locus}, $L$ is nef.
\end{remark}

We now show that Theorem~\ref{thm:nef-on-stable-base-locus} also
holds for arithmetically nef. Note that all properties discussed
in the proof are stable under a base change $\Spec R_\beta \to \Spec R_\alpha$,
so for simplicity we simply speak of choosing $\alpha$ such 
that a property holds over $\Spec R_\alpha$.

\begin{proof}[Proof of Theorem~\ref{thm:arith-nef-on-stable-base-locus}]
Let $Z = \SBs(L)$.
Suppose $L$ is arithmetically nef. Then $L\vert_{Z}$
is arithmetically nef by Lemma~\ref{lem:arith-nef-surjective}.

Now assume $L\vert_{Z}$ is arithmetically nef. 
Replacing $L$ with a positive power, we may assume $Z = \Bs(L)$.
By Definition~\ref{def:base-locus}
 of base locus, there exists $m > 0$ and a right exact sequence
\begin{equation}\label{eq:right-exact}
\oplus_{i=1}^m \OO_X \to L \to F \to 0
\end{equation}
with $\Supp (F) = Z$. 

As described in Section~\ref{sec:arith-nef}, we can choose a thickening
with $X_\alpha \to \Spec R_\alpha$ proper and $L_\alpha$ invertible.
We can also thicken the sequence \eqref{eq:right-exact} to
the right exact 
\begin{equation}\label{eq:right-exact-thick}
\oplus_{i=1}^m \OO_{X_\alpha} \to L_\alpha \to F_\alpha \to 0
\end{equation}
by \cite[$\mathrm{IV}_3$, 8.5.6]{ega}. Let $Z_\alpha = \Supp(F_\alpha)$
and $i: X \to X_\alpha$
be the morphism induced by the base change $\Spec k \to \Spec R_\alpha$.
Then $F = i^*F_\alpha$ and $Z = \Supp(F) = i^{-1}Z_\alpha$
\cite[Tag 056J]{stacks-project}. By hypothesis and the definition
of arithmetically nef, we can also choose $\alpha$ so that
$L_\alpha\vert_{Z_\alpha}$ is nef. That is, if $s \in \Spec R_\alpha$
is a closed point, and $Z_s, X_s$ are the fibers (respectively) of 
$Z_\alpha,X_\alpha$ over $s$, then $L_\alpha\vert_{Z_s}$ is nef.

Let $i_s: X_s \to X_\alpha$ be the natural closed immersion. We can pullback
\eqref{eq:right-exact-thick} to the fiber $X_s$. Write $L_s = i_s^*L_\alpha$
and $F_s = i_s^*F_\alpha$. Then $\Supp(F_s) = i_s^{-1}Z_\alpha = Z_s$.
By definition of (stable) base locus, $\SBs(L_s) \subseteq \Bs(L_s) \subseteq Z_s$. 
Since $L_s\vert_{Z_s} \cong L_\alpha\vert_{Z_s}$ is nef, 
we have $L_s\vert_{\SBs(L_s)}$ nef. Hence $L_s$ is nef by 
Theorem~\ref{thm:nef-on-stable-base-locus}. 
Since this holds for all closed $s \in \Spec R_\alpha$,
we have that $L_\alpha$ is nef by definition, and hence $L$ is arithmetically
nef by definition.
\end{proof}

Now combining Corollary~\ref{cor:arith-nef-on-curves} and 
Theorem~\ref{thm:arith-nef-on-stable-base-locus} we immediately have the following.

\begin{corollary}\label{cor:small-stable-base-locus}
Let $X$ be a proper scheme over a field $k$ with line bundle $L$.
If $L$ is nef and $\dim \SBs(L) \leq 1$, then $L$ is 
arithmetically nef. 

In particular, if $X$ is an integral surface and $L$ is nef and $L^m$ is effective
for some $m > 0$, then $L$ is arithmetically nef. \qed
\end{corollary}

\section{Counterexamples}\label{sec:counterexamples}

In this section, we consider examples of line bundles $L$ which
are nef, but not arithmetically nef.
In \cite[Section~8]{Langer1}, Langer gave an example of a nef, but not arithmetically
nef, line bundle on a smooth projective scheme over a field of characteristic $0$.
As he works over $\ZZ[\frac{1}{N}]$ for some natural number $N$, it
is clear that his example is not arithmetically nef.


We now consider a characteristic $p > 0$ example
and verify that the line bundle is nef, but not arithmetically nef.

\begin{example}\label{ex:char-p}
In \cite[Example~5.3]{Langer2}, pulling together
the results of a few authors, Langer gives an example of a smooth,
projective morphism $\pi: X \to S$ with $S = \mathbb{A}^1_k$ and 
$k = \overline{\mathbb{F}}_2$. There exists a line bundle $L$ on $X$
such that $L$ is nef on the generic fiber, but not on any closed fiber.
This will also be a counterexample to arithmetic nefness as follows.

Let $\{ k_\alpha \}$ be the directed system of finite subfields of $k$
and $R_\alpha = k_\alpha[t]$. By Remark~\ref{rem:char-p}, the $k_\alpha$
are all the finite type $\ZZ$-subalgebras of $k$.
 Since all schemes are finite type
over $k$, we have arithmetic thickenings $X_\alpha, L_\alpha$
over $k_\alpha$. For $\alpha$ sufficiently large,
$\pi_\alpha: X_\alpha \to S_\alpha = \Spec R_\alpha$ is smooth and projective
\cite[$\mathrm{IV}_{4,3}$, 17.7.8, 8.10.5]{ega}.

Since $k$ is algebraic over $k_\alpha$, we have $k[t]$ integral over $k_\alpha[t]$.
Thus $\Spec k[t] \to \Spec k_\alpha[t]$ is surjective
\cite[{Tag 00GQ}]{stacks-project}
and closed points map to closed points \cite[{Tag 00GT}]{stacks-project}.
Thus $L_\alpha$ is nef on the generic fiber of $\pi_\alpha$,
but not nef on any closed fiber \cite[Lemma~2.18]{keeler}. 
The $R_\alpha$ are part of the directed system of finitely generated $\ZZ$-subalgebras
of $K = k(t)$. By Corollary~\ref{cor:non-arith-nef},
 $L$ restricted to the generic fiber of $\pi$ is an example
of a nef, but not arithmetically nef, line bundle over a field $K$
of positive characteristic.
\end{example}

\section*{Acknowledgements}
We thank D.~Arapura, A.~Langer, and D.~Litt for discussions on this topic.


\providecommand{\bysame}{\leavevmode\hbox to3em{\hrulefill}\thinspace}
\providecommand{\MR}{\relax\ifhmode\unskip\space\fi MR }
\providecommand{\MRhref}[2]{%
  \href{http://www.ams.org/mathscinet-getitem?mr=#1}{#2}
}
\providecommand{\href}[2]{#2}

\expandafter\ifx\csname url\endcsname\relax
  \def\url#1{\texttt{#1}}\fi
\expandafter\ifx\csname urlprefix\endcsname\relax\def\urlprefix{URL }\fi
\expandafter\ifx\csname href\endcsname\relax
  \def\href#1#2{#2} \def\path#1{#1}\fi

\end{document}